\documentclass{article}
\usepackage{amsfonts}
\usepackage{amsmath}
\usepackage{amsthm}   
\usepackage{amssymb}
\usepackage{graphicx}
\usepackage{float}
\usepackage{booktabs}
\usepackage{fancybox}
\usepackage{fancyhdr}
\usepackage{tabularx}
\usepackage[a4paper,top=2cm,bottom=2cm]{geometry}
\usepackage{authblk}
\usepackage[numbers]{natbib}
\newtheorem{theorem}{Theorem}[section]

\newtheorem{proposition}[theorem]{Proposition}
\newtheorem{corollary}[theorem]{Corollary}
\newtheorem{assumption}[theorem]{Assumption}

\theoremstyle{definition}

\theoremstyle{remark}
\newtheorem{remark}[theorem]{Remark}
\usepackage{listings}
\usepackage[colorlinks=true, linkcolor=black, citecolor=blue, urlcolor=blue]{hyperref}
\title{Tikhonov-Regularized Physics-Informed Neural Networks for Terminal-State Distributed Optimal Control of Parabolic Partial Differential Equations}

\author[1]{Nguyen Thanh Quang\thanks{quang.nt227031@sis.hust.edu.vn}}
\author[1]{Ta Thi Thanh Mai\thanks{mai.tathithanh@hust.edu.vn}}
\affil[1]{Faculty of Mathematics and Informatics, Hanoi University of Science and Technology, Ha Noi, Viet Nam}

\date{\today}
\begin{document}
    \maketitle
    \begin{abstract}
        In this study, we propose a Physics-Informed Neural Networks 
        (PINNs) framework that incorporates Tikhonov regularization to solve 
        terminal-state tracking optimal control constrained by parabolic 
        partial differential equations (PDEs).
        This problem  is inherently ill-posed, as infinitely many distributed
        controls may drive the system to the same desired state, so the regularization 
        guides the optimizer toward the minimum-energy control, understood here as the control with the smallest space-time $H^1$ norm, restoring numerical 
        stability and yielding a smooth, physically meaningful solution. 
        On the theoretical side, we establish 
        a consistency result showing that PINNs minimizers nearly attain 
        the continuous regularized objective under residual and quadrature 
        approximation assumptions, and a novel error estimate that 
        bounds the deviation of the learned control from the 
        minimum-energy solution in terms of the PINNs training residuals 
        and the regularization parameter. Numerical experiments on the 
        linear heat equation and the nonlinear Burgers' equation demonstrate 
        that the regularized PINNs framework accurately achieves the target terminal state 
        while producing controls that have significantly lower energy and smoother profiles 
        than their unregularized counterparts.
    \end{abstract}
    \section{Introduction}

Optimal control of systems governed by partial differential equations (PDEs) plays a central role in engineering and scientific computing, with applications in thermal regulation, fluid mechanics, environmental management, and biomedicine~\cite{badescu2017optimal, braack2009optimal, shastri2008optimal, swan1981optimal}.

A particularly challenging class is the \emph{terminal-state optimal control problem}. Its aim is to choose a control distributed over space and time so that the solution of the governing PDE is as close as possible to a prescribed target at the final time. The discrepancy between the resulting terminal state and the target state defines the objective to be minimized, while the PDE, boundary conditions, and initial condition act as constraints~\cite{logist2004optimal,hou2007analysis}. 

A fundamental difficulty is the inherent \emph{ill-posedness} of this problem~\cite{serovaiskii2011counterexamples, d2023qualitative}: due to the smoothing properties of parabolic PDEs~\cite{troltzsch2010optimal}, infinitely many controls may produce the same terminal state, violating Hadamard's uniqueness condition.
Tikhonov regularization is a classical remedy~\cite{engl1996regularization, golub1999tikhonov}. When the target is attainable, it selects, among all controls achieving the target, the one with minimum space-time $H^1$ norm, thereby favoring smooth controls without spurious oscillations~\cite{marengo2002inverse, naidu2018optimal}. Crucially, this penalty is not part of the original terminal-tracking objective but is introduced at the solver level~\cite{casas2018influence}, preserving the control objective while restoring coercivity and numerical stability. Classical theory guarantees that, as the strength of the regularization is progressively reduced, the regularized solutions approach the minimum-energy control~\cite{engl1996regularization, laszlo2025proximal}.

Traditional mesh-based, adjoint-based methods for PDE-constrained control~\cite{de2015numerical} are mathematically rigorous but require specialized expertise and become costly in high dimensions or for nonlinear problems. Physics-Informed Neural Networks (PINNs)~\cite{raissi2019physics} offer a mesh-free alternative that embeds PDE residuals directly into the loss function, bypassing explicit mesh generation and adjoint derivations. PINNs have been applied to forward problems such as elasticity~\cite{hoang2026gradnorm}, inverse problems such as Navier-Stokes parameter identification~\cite{pham2025physics}, impedance imaging~\cite{duan2024current}, and parabolic inverse source problems~\cite{zhang2023stability}. Their application to optimal control has gained considerable attention~\cite{mowlavi2023optimal, lai2025hard, li2026aw, song2024admm, ji2025potential}. However, a systematic integration of Tikhonov regularization within the PINNs loss for terminal-state control, together with principled parameter selection, remains underexplored.

In this paper, we propose a Tikhonov-regularized PINNs framework for terminal-state tracking distributed optimal control of parabolic PDEs. Our contributions are: (i) a PINNs loss incorporating a space-time $H^1$ Tikhonov penalty that guides the optimizer toward minimum-energy solutions without modifying the control objective; (ii) a consistency estimate (Proposition~\ref{prop:consistency}) and, under a standard source condition, a novel error bound (Theorem~\ref{thm:control_error}) that relates the accuracy of the learned control to the training accuracy and regularization strength; and (iii) numerical validation on the heat and Burgers' equations using the L-curve criterion~\cite{hansen1992analysis,calvetti2000tikhonov} for parameter selection.

The remainder of this paper is structured as follows. 
Section~2 presents the precise mathematical formulation of the terminal-state optimal control problem, including the governing parabolic PDE, the terminal-cost functional, and a brief discussion on the inherent ill-posedness caused by the non-uniqueness of distributed controls. 
Section~3 describes our Tikhonov-regularized Physics-Informed Neural Networks (PINNs) framework, in which both the state and the control are approximated by independent neural networks, and the loss function incorporates not only the PDE residuals but also a space-time $H^1$ penalty on the control to enforce smoothness and select the minimum-energy solution. 
Section~4 provides the theoretical analysis of the proposed method: we first prove a consistency result for the regularized objective when the residuals are small, and then derive a novel error estimate for the learned control in the space-time $H^1$ norm in terms of the training residual and the regularization parameter; the corresponding space-time $L^2$ estimate follows by continuous embedding. 
Section~5 reports numerical experiments on two benchmark problems-the linear heat equation and the nonlinear Burgers' equation-to illustrate the effectiveness of the regularized PINNs; we compare its performance with that of the corresponding unregularized formulation and also demonstrate the use of the L-curve criterion for selecting the regularization parameter. 
Finally, Section~6 concludes the paper with a summary of our main contributions and a discussion of several promising directions for future work, such as extending the framework to boundary control, adaptive loss-weighting, and rigorous error analysis for fully nonlinear problems.

    \section{Problem Formulation}

\subsection{Governing Equations}

Let $\Omega \subset \mathbb{R}^d$ be a bounded domain with smooth boundary $\partial\Omega$ and $T > 0$ a fixed final time. We write $\Omega_T:=\Omega\times(0,T)$ for the space-time cylinder and $\Sigma_T:=\partial\Omega\times(0,T)$ for its lateral boundary. For any domain $D$, $L^2(D)$ denotes the space of square-integrable functions on $D$, while $H^k(D)$, for $k=1,2$, denotes the Sobolev space of functions with square-integrable weak derivatives up to order $k$. Their norms are denoted by $\|\cdot\|_{L^2(D)}$ and $\|\cdot\|_{H^k(D)}$, respectively. We use $\nabla_{x,t}$ for the gradient with respect to all space-time variables and $\Delta$ for the spatial Laplacian. The state $u(x,t)$ is governed by:
\begin{equation}\label{eq:state_pde}
\begin{cases}
\partial_t u(x,t) + \mathcal{F}[u(x,t)] = f(x,t), & (x,t) \in \Omega_T, \\
\mathcal{B}[u(x,t)] = 0, & (x,t) \in \Sigma_T, \\
u(x,0) = u_0(x), & x \in \Omega,
\end{cases}
\end{equation}
where $\mathcal{F}$ is a (possibly nonlinear) spatial differential operator, $\mathcal{B}$ encodes homogeneous boundary conditions, and $u_0\in L^2(\Omega)$ is the initial state. Let $V\subset H^1(\Omega)$ be the energy space incorporating the boundary condition encoded by $\mathcal{B}$, and assume that $V$ is densely and continuously embedded in $L^2(\Omega)$; for example, $V=H_0^1(\Omega)$ for homogeneous Dirichlet conditions, where $H_0^1(\Omega)$ consists of $H^1(\Omega)$ functions with zero boundary trace. Denote the dual of $V$ by $V^*$ and define:
\[
W(0,T;V):=\bigl\{v\in L^2(0,T;V):\partial_t v\in L^2(0,T;V^*)\bigr\}.
\]
For every control $f\in H^1(\Omega_T)$, we assume that \eqref{eq:state_pde} admits a unique weak solution $u_f\in W(0,T;V)$. Since
\[
W(0,T;V)\hookrightarrow C([0,T];L^2(\Omega)),
\]
the terminal value $u_f(\cdot,T)$ is well defined. The map $f\mapsto u_f$ is the \textit{control-to-state map}, and we use the equivalent notations $u_f$ and $u(f)$ for this state.

\subsection{Terminal-State Optimal Control Problem}

The goal is to find a distributed control $f^*$ minimizing the terminal-cost functional
\begin{equation}\label{eq:cost}
J(f) = \frac{1}{2}\int_\Omega \bigl|u_f(x,T) - u_d(x)\bigr|^2\,dx
\end{equation}
subject to \eqref{eq:state_pde}, where $u_d \in L^2(\Omega)$ is the desired target state at the terminal time $T$:
\begin{equation}\label{eq:ocp}
f^* \in \arg\min_{f\in H^1(\Omega_T)} J(f).
\end{equation}

\subsection{Ill-Posedness and the Minimum-Energy Criterion}

Problem \eqref{eq:ocp} can be ill-posed in the sense of Hadamard. For linear parabolic PDEs, the control-to-terminal-state map is compact and, for distributed controls on $\Omega_T$, generally non-injective; controls that differ by an element of its kernel yield the same terminal state. For nonlinear PDEs, a linear null space is not defined; instead, non-convexity of the terminal-cost functional may produce multiple solution branches and unstable control recovery.

Consequently, when several controls attain the same minimum value of $J$, the terminal-tracking objective alone cannot distinguish among them. We therefore introduce an additional selection criterion: among the minimizers of $J$, we choose a control with the smallest $H^1(\Omega_T)$ norm and call it the \emph{minimum-energy control}. If the target is attainable, this is equivalently the smallest-$H^1(\Omega_T)$ control that reaches the target exactly. The criterion favors smooth controls and suppresses spurious oscillations, but it is not part of the original objective $J$; instead, it is implemented through the Tikhonov penalty introduced in Section~3. For the linear problem, standard inverse-problem theory shows that, under the usual regularity assumptions, the regularized minimizers approach this minimum-energy control as the penalty strength tends to zero~\cite{engl1996regularization}.

    \section{Tikhonov-Regularized PINNs Framework}

We parameterize both the state $u(x,t)$ and the distributed control $f(x,t)$ as independent fully connected feedforward neural networks:
\begin{align}
u(x,t) \approx u_{\theta}(x,t), \qquad f(x,t) \approx f_{\eta}(x,t),
\end{align}
where $\theta\in\mathbb{R}^{p_u}$ and $\eta\in\mathbb{R}^{p_f}$ collect the weights and biases of the state and control networks, respectively; $p_u$ and $p_f$ are the corresponding numbers of trainable parameters. Both networks take space-time coordinates $(x,t)\in\mathbb{R}^{d+1}$ as input and use the $\tanh$ activation function, which is smooth and supports higher-order derivative computation via automatic differentiation. The network $u_\theta$ has $L_u$ hidden layers with $N_u$ neurons per layer; $f_\eta$ has $L_f$ hidden layers with $N_f$ neurons per layer. Both are trained jointly by minimizing a single shared loss.

\subsection{Loss Function with Tikhonov Regularization}

We use the following sample sets:
\begin{align*}
\{(x_i^{\mathrm{int}},t_i^{\mathrm{int}})\}_{i=1}^{N_{\mathrm{int}}}&\subset\Omega_T,
&\{(x_i^{\mathrm{sb}},t_i^{\mathrm{sb}})\}_{i=1}^{N_{\mathrm{sb}}}&\subset\Sigma_T,\\
\{x_i^{\mathrm{tb}}\}_{i=1}^{N_{\mathrm{tb}}}&\subset\Omega,
&\{x_i^{\mathrm{d}}\}_{i=1}^{N_d}&\subset\Omega,\\
\{(x_i^{\mathrm{reg}},t_i^{\mathrm{reg}})\}_{i=1}^{N_{\mathrm{reg}}}&\subset\Omega_T.
\end{align*}
They contain the interior, spatial-boundary, initial-condition, desired-terminal-state, and regularization sample points, respectively. The subscripts ``int,'' ``sb,'' ``tb,'' ``d,'' and ``reg'' refer to these five sample sets. We define the interior, spatial-boundary, and initial-condition residual functions by
\begin{align}
\mathcal{R}_{\mathrm{int}}(x,t;\theta,\eta)
&:= \partial_t u_\theta(x,t)+\mathcal{F}[u_\theta](x,t)-f_\eta(x,t),
&& (x,t)\in\Omega_T, \label{eq:res_int}\\
\mathcal{R}_{\mathrm{sb}}(x,t;\theta)
&:= \mathcal{B}[u_\theta](x,t),
&& (x,t)\in\Sigma_T, \label{eq:res_sb}\\
\mathcal{R}_{\mathrm{tb}}(x;\theta)
&:= u_\theta(x,0)-u_0(x),
&& x\in\Omega. \label{eq:res_tb}
\end{align}
The total loss is:
\begin{equation}\label{eq:loss}
\mathcal{L}(\theta,\eta) = \mathcal{L}_{\text{int}} + \mathcal{L}_{\text{sb}} + \mathcal{L}_{\text{tb}} + w_J\,\mathcal{L}_{\text{target}} + \alpha\,\mathcal{L}_{\text{reg}},
\end{equation}
where $w_J>0$ weights the control objective and $\alpha>0$ is the regularization parameter. The individual terms are:
\begin{align}
\mathcal{L}_{\text{int}} &= \frac{1}{N_{\text{int}}}\sum_{i=1}^{N_{\text{int}}} \bigl|\mathcal{R}_{\mathrm{int}}(x_i^{\text{int}},t_i^{\text{int}};\theta,\eta)\bigr|^2, \label{eq:loss_int}\\
\mathcal{L}_{\text{sb}} &= \frac{1}{N_{\text{sb}}}\sum_{i=1}^{N_{\text{sb}}} \bigl|\mathcal{R}_{\mathrm{sb}}(x_i^{\text{sb}},t_i^{\text{sb}};\theta)\bigr|^2, \label{eq:loss_sb}\\
\mathcal{L}_{\text{tb}} &= \frac{1}{N_{\text{tb}}}\sum_{i=1}^{N_{\text{tb}}} \bigl|\mathcal{R}_{\mathrm{tb}}(x_i^{\text{tb}};\theta)\bigr|^2, \label{eq:loss_tb}\\
\mathcal{L}_{\text{target}} &= \frac{1}{2N_d}\sum_{i=1}^{N_d} \bigl|u_\theta(x_i^{\mathrm d},T) - u_d(x_i^{\mathrm d})\bigr|^2, \label{eq:loss_target}\\
\mathcal{L}_{\text{reg}} &= \frac{1}{2N_{\text{reg}}}\sum_{i=1}^{N_{\text{reg}}} \Bigl[|f_\eta|^2 + |\nabla_{x,t}f_\eta|^2\Bigr]_{(x_i^{\text{reg}},t_i^{\text{reg}})}. \label{eq:loss_reg}
\end{align}
All derivatives are computed via automatic differentiation. When $\alpha=0$, \eqref{eq:loss} reduces to the standard unregularized PINNs formulation, which is ill-posed and prone to non-unique or oscillatory solutions. The term $\mathcal{L}_{\text{reg}}$ with $\alpha>0$ restores coercivity and stabilizes training.

Thus, \(\mathcal{L}_{\mathrm{int}}\), \(\mathcal{L}_{\mathrm{sb}}\), and \(\mathcal{L}_{\mathrm{tb}}\) are empirical quadrature approximations of the squared continuous norms \(\|\mathcal{R}_{\mathrm{int}}\|_{L^2(\Omega_T)}^2\), \(\|\mathcal{R}_{\mathrm{sb}}\|_{L^2(\Sigma_T)}^2\), and \(\|\mathcal{R}_{\mathrm{tb}}\|_{L^2(\Omega)}^2\), respectively. In the analysis, residuals written without explicit parameter arguments are evaluated at the trained parameters $(\theta^*,\eta^*)$. Whenever a continuous residual bound is used, we state it as a separate assumption; it does not follow from a small empirical loss without an additional quadrature or generalization estimate.

\subsection{Training Strategy and Hyperparameter Selection}

Training seeks parameters
\[
(\theta^*,\eta^*)\in\arg\min_{(\theta,\eta)\in\mathbb{R}^{p_u}\times\mathbb{R}^{p_f}}\mathcal{L}(\theta,\eta)
\]
with the Adam stochastic-gradient optimizer, using a learning rate decayed by a fixed factor at prescribed epoch milestones. Collocation points are resampled periodically to improve domain coverage.

The weight $w_J$ is selected by a line search over positive values on a logarithmic scale. For each candidate value, the networks are trained and the learned control is inserted into the governing PDE to compute the rollout state. We then evaluate the rollout terminal objective
\[
J_{\mathrm{roll}}(w_J):=\frac12\bigl\|u(f_{\eta^*(w_J)})(\cdot,T)-u_d\bigr\|_{L^2(\Omega)}^2,
\]
where $\eta^*(w_J)$ denotes the trained control-network parameters for that candidate. The value of $w_J$ giving the smallest rollout objective is selected; using the rollout rather than $u_{\theta^*}$ alone also checks that terminal accuracy is consistent with the governing PDE. The regularization parameter $\alpha$ is chosen separately via the L-curve criterion~\cite{hansen1992analysis,calvetti2000tikhonov}: we train for a range of $\alpha$ values and identify the point of maximum curvature on the log-log plot of $\mathcal{L}_{\text{target}}$ versus $\mathcal{L}_{\text{reg}}$, corresponding to the optimal trade-off between data fidelity and control regularity.

    \section{Theoretical Analysis: Ill-Posedness, Regularization, and Error Estimates}
\label{sec:theory}

In this section, we provide a rigorous mathematical discussion of the terminal-state optimal control problem.
We first recall the inherent ill-posedness due to non-uniqueness (Section~\ref{subsec:illposed}).
We then analyze the stabilizing effect of the Tikhonov regularization on the control network (Section~\ref{subsec:stabilizing}) and establish a quantitative connection to the minimum-energy control (Section~\ref{subsec:connection}).
We next present a consistency result (Proposition~\ref{prop:consistency}) showing that the PINNs minimizer nearly attains the optimal regularized objective.
Finally, we prove a new error estimate (Theorem~\ref{thm:control_error}) that directly bounds the deviation of the learned control from the minimum-energy solution, using only the residual losses of the PINNs training.

\subsection{Ill-Posedness of the Unregularized Problem}
\label{subsec:illposed}

Consider the control-to-terminal-state map $\mathcal{S}: f \mapsto u_f(\cdot,T)$, which assigns to each distributed control $f \in H^1(\Omega_T)$ the terminal state of the PDE constraint~\eqref{eq:state_pde}.
For linear parabolic equations (e.g., the heat equation), $\mathcal{S}:H^1(\Omega_T)\to L^2(\Omega)$ is a compact linear operator \cite{vrabie2002compactness}.
Compactness alone does not imply non-injectivity. For the distributed-in-time control setting considered here, however, each terminal-state mode depends on a time integral of the corresponding control mode, leaving an infinite-dimensional kernel $\mathcal{N}(\mathcal{S})$; hence there are infinitely many non-zero controls $f_{\rm null}$ with $\mathcal{S}(f_{\rm null})=0$.
If $f^*$ is any control achieving the desired terminal state $u_d$, then the entire affine set:
\[
f^* + \mathcal{N}(\mathcal{S}) = \{f^* + f_{\rm null} \mid f_{\rm null}\in\mathcal{N}(\mathcal{S})\}
\]
also yields the same terminal state and therefore attains the same cost value $J(f^*)$.
Hence the unregularized problem~\eqref{eq:ocp} admits infinitely many global minimizers, violating Hadamard's uniqueness condition.

For nonlinear PDEs such as the Burgers' equation, the control-to-terminal-state map is nonlinear, so the affine null-space argument above no longer applies. Instead, composing this nonlinear map with the terminal-tracking objective can make $J(f)$ non-convex, potentially producing multiple local minimizers or distinct solution branches~\cite{d2023qualitative}.
Gradient-based optimization applied to such problems often exhibits instability and converges to controls contaminated by high-frequency oscillations, reflecting the underlying non-uniqueness and the lack of coercivity.

\subsection{Stabilizing Effect of the Tikhonov Regularization}
\label{subsec:stabilizing}

Although the original optimal control problem~\eqref{eq:ocp} does not contain a regularization term, the PINNs loss function~\eqref{eq:loss} incorporates an explicit $H^1(\Omega_T)$ Tikhonov penalty $\alpha \mathcal{L}_{\rm reg}$ on the control network $f_\eta$.
This modification fundamentally alters the optimization landscape and is essential for obtaining physically meaningful solutions.

When $\mathcal{F}$ is a linear elliptic operator (e.g., $\mathcal{F}[u] = -\Delta u$), the reduced terminal-cost functional is convex in the control function $f$.
Adding the squared $H^1(\Omega_T)$ penalty makes the continuous reduced functional strongly convex and coercive on $H^1(\Omega_T)$; hence, for every fixed $\alpha>0$, it admits a unique minimizing control $f_\alpha$.
This uniqueness holds for the optimal control function, not necessarily for the neural-network parameters used to represent it. PINNs training optimizes the weights $(\theta,\eta)$, and the network outputs depend nonlinearly on these weights. Moreover, neuron permutations and other parameter symmetries allow different parameter pairs to represent the same state and control functions. Consequently, even when the continuous regularized problem has a unique optimal control $f_\alpha$, the PINNs loss $\mathcal{L}(\theta,\eta)$ is generally non-convex and need not have a unique parameter minimizer.
Under the assumptions stated in Section~\ref{subsec:connection}, the continuous minimizers converge strongly in $H^1(\Omega_T)$ to the minimum-energy control as $\alpha\to0^+$.

For nonlinear PDEs, the reduced terminal-cost functional is generally non-convex.
At the function-space level, the $H^1(\Omega_T)$ penalty is coercive and weakly lower semicontinuous with respect to the control function $f$ \cite{engl1996regularization}. Together with suitable weak continuity or closedness assumptions on the nonlinear control-to-terminal-state map, these properties yield existence of a minimizing control, but not uniqueness.
They do not by themselves guarantee existence or uniqueness in the raw neural-network parameter space. In computation, the penalty nevertheless suppresses highly oscillatory control functions and favors smoother local minima.

\subsection{Connection to the Minimum-Energy Control}
\label{subsec:connection}

Assume that the target state $u_d$ is attainable, i.e., there exists at least one admissible control $f^\dagger\in H^1(\Omega_T)$ such that $u_{f^\dagger}(\cdot,T) = u_d$.
The \emph{minimum-energy control} $f_{\min}$ is defined by:
\begin{equation}\label{eq:minenergy}
  \begin{aligned}
    &f_{\min} = \arg\min_{f\in H^1(\Omega_T)} \|f\|_{H^1(\Omega_T)} \\
    &\quad \text{subject to} \quad u_f(\cdot,T) = u_d.
  \end{aligned}
\end{equation}
For a linear control-to-terminal-state map, the feasible set is closed and affine, so strict convexity of the squared $H^1(\Omega_T)$ norm makes this control unique. For nonlinear problems, uniqueness is an additional assumption and $f_{\min}$ denotes a selected minimum-energy solution.
In the linear case, classical regularization theory \cite{engl1996regularization} establishes that the minimizers $f_\alpha$ of the regularized problem:
\[
\min_{f\in H^1(\Omega_T)} \left( \frac12\|u_f(\cdot,T)-u_d\|_{L^2(\Omega)}^2 + \frac{\alpha}{2}\|f\|_{H^1(\Omega_T)}^2 \right)
\]
converge strongly in $H^1(\Omega_T)$ to $f_{\min}$ as $\alpha\to0^+$. Analogous convergence for a nonlinear problem requires additional compactness, stability, and uniqueness assumptions.
In our PINNs framework, a small but fixed $\alpha>0$ is employed during training.
Consequently, the computed control $f_{\eta^*}$ serves as a stable approximation of the minimum-energy solution.
This connection justifies the Tikhonov penalty not merely as a numerical trick, but as a principled mechanism to select the physically most favorable control among infinitely many mathematically admissible ones.

\subsection{Consistency of the PINNs Minimizer}
\label{subsec:consistency}

The following proposition compares the complete regularized objective, rather than its terminal component alone. This distinction is necessary because minimizing a sum of terminal and regularization terms does not imply that either term is minimized separately.

\begin{proposition}[Consistency of the PINNs minimizer]\label{prop:consistency}
Fix $\alpha_0\ge0$, and let
\[
f_{\alpha_0}\in\arg\min_{f\in H^1(\Omega_T)}
\left[
\frac{w_J}{2}\|u(f)(\cdot,T)-u_d\|_{L^2(\Omega)}^2
+\frac{\alpha_0}{2}\|f\|_{H^1(\Omega_T)}^2
\right]
\]
be a minimizer of the continuous weighted regularized problem. Define its optimal value by
\[
\mathcal{J}_{\alpha_0}^*
=\frac{w_J}{2}\|u(f_{\alpha_0})(\cdot,T)-u_d\|_{L^2(\Omega)}^2
+\frac{\alpha_0}{2}\|f_{\alpha_0}\|_{H^1(\Omega_T)}^2.
\]
Let $(u_{\theta^*},f_{\eta^*})$ be a global minimizer of \eqref{eq:loss} with $\alpha=\alpha_0$. Suppose there is a comparison pair $(u_{\tilde\theta},f_{\tilde\eta})$ approximating $(u(f_{\alpha_0}),f_{\alpha_0})$ such that its three constraint losses sum to at most $\epsilon$ and its sampled target and regularization terms approximate their continuous counterparts with total weighted error at most $C_{\mathrm{quad}}\sqrt{\epsilon}$. Here $C_{\mathrm{quad}}>0$ is a quadrature-accuracy constant independent of $\epsilon$. Then, for $0<\epsilon\le1$,
\[
w_J\mathcal{L}_{\rm target}(\theta^*)
+\alpha_0\mathcal{L}_{\rm reg}(\eta^*)
\le \mathcal{J}_{\alpha_0}^*+C'\sqrt{\epsilon},
\]
where $C'=C_{\mathrm{quad}}+1$ is independent of $\epsilon$.
\end{proposition}

\begin{proof}
The three constraint-loss terms in the total loss are non-negative, and global minimality gives
\[
\begin{aligned}
w_J\mathcal{L}_{\rm target}(\theta^*)
+\alpha_0\mathcal{L}_{\rm reg}(\eta^*)
&\le \mathcal{L}(\theta^*,\eta^*)\\
&\le \mathcal{L}(\tilde\theta,\tilde\eta)\\
&\le \mathcal{J}_{\alpha_0}^*+C_{\mathrm{quad}}\sqrt{\epsilon}+\epsilon.
\end{aligned}
\]
Since $\epsilon\le\sqrt{\epsilon}$, the conclusion follows. In particular, the regularization term has been retained on both sides; dropping the non-negative comparison term $\alpha_0\mathcal{L}_{\rm reg}(\tilde\eta)$ would not be valid.
\end{proof}

\subsection{Error Estimate for the Learned Control}
\label{subsec:error}

While Proposition~\ref{prop:consistency} controls the regularized objective value, it does not quantify how close the computed control $f_{\eta^*}$ is to the minimum-energy solution $f_{\min}$.
We address this question first for linear parabolic equations, where a rigorous estimate can be derived using the strong convexity of the Tikhonov functional.
A heuristic extension to nonlinear problems is discussed afterwards.

\subsubsection{Tikhonov functional and strong convexity (linear case)}
Assume that the spatial operator $\mathcal{F}$ is linear (e.g.\ $\mathcal{F}[u] = -\Delta u$), so that the control-to-terminal-state map:
\[
\mathcal{S}: f \mapsto u_f(\cdot,T), \qquad H^1(\Omega_T) \to L^2(\Omega)
\]
is linear and compact~\cite{vrabie2002compactness}.
For a fixed regularization parameter $\alpha>0$, define the continuous Tikhonov functional:
\begin{equation}\label{eq:tikhonov}
\mathcal{J}_\alpha(f) = \frac12\|\mathcal{S}f - u_d\|_{L^2(\Omega)}^2 + \frac{\alpha}{2}\|f\|_{H^1(\Omega_T)}^2.
\end{equation}
The two norms in \eqref{eq:tikhonov} have different roles. The terminal state and target belong to $L^2(\Omega)$, so their tracking error is measured in the state-space norm $\|\cdot\|_{L^2(\Omega)}$. The control belongs to $H^1(\Omega_T)$, and the Tikhonov term uses $\|\cdot\|_{H^1(\Omega_T)}$ to penalize both its magnitude and its space-time variation. Thus, the use of these two spaces in the same functional is intentional rather than an interchange of norms.

\begin{remark}[Continuous and sampled regularization]\label{rem:L2_vs_H1}
We use the same $H^1(\Omega_T)$ norm as in the numerical loss. The sampled quantity $\mathcal{L}_{\rm reg}$ is a quadrature approximation of the continuous squared $H^1(\Omega_T)$ norm; fixed normalization factors can be absorbed into $\alpha$. The estimates below therefore require the stated quadrature-accuracy assumption. An $L^2(\Omega_T)$ growth bound alone would not imply an $H^1(\Omega_T)$ growth bound, because the continuous embedding gives the opposite implication: for every $v\in H^1(\Omega_T)$, $\|v\|_{L^2(\Omega_T)}\le C_{\mathrm{emb}}\|v\|_{H^1(\Omega_T)}$, where $C_{\mathrm{emb}}>0$ is the embedding constant.
\end{remark}

Since $\frac12\|\mathcal{S}f-u_d\|_{L^2(\Omega)}^2$ is convex and $\frac{\alpha}{2}\|f\|_{H^1(\Omega_T)}^2$ is strongly convex with constant $\alpha$, the functional $\mathcal{J}_\alpha$ is $\alpha$-strongly convex on $H^1(\Omega_T)$.
In particular, its unique global minimizer:
 \[f_\alpha \in \arg\min_{f\in H^1(\Omega_T)} \mathcal{J}_\alpha(f)\] 
satisfies the quadratic growth condition:
\begin{equation}\label{eq:growth}
\mathcal{J}_\alpha(f) - \mathcal{J}_\alpha(f_\alpha) \ge \frac{\alpha}{2}\|f-f_\alpha\|_{H^1(\Omega_T)}^2, \qquad \forall f\in H^1(\Omega_T).
\end{equation}

\subsubsection{Error estimate for linear problems}

To quantify the PINNs error, we state explicitly the approximation property required in the proof.

\begin{assumption}[Network expressivity]\label{ass:express}
Let $(u,f)$ be any admissible state--control pair satisfying the linear PDE~\eqref{eq:state_pde}, with $u\in H^2(\Omega_T)$ and $f\in H^1(\Omega_T)$. For every prescribed tolerance $\delta>0$, assume that the state and control network classes contain a comparison pair $(u_{\tilde\theta},f_{\tilde\eta})$. Equivalently, there exist parameters $\tilde\theta=\tilde\theta(u,\delta)$ and $\tilde\eta=\tilde\eta(f,\delta)$ such that
\[
\|u_{\tilde{\theta}} - u\|_{H^1(\Omega_T)} \le \delta,\qquad
\|f_{\tilde{\eta}} - f\|_{H^1(\Omega_T)} \le \delta,
\]
and their empirical constraint residuals satisfy
\[
\mathcal{L}_{\mathrm{int}}(\tilde{\theta},\tilde{\eta}) + \mathcal{L}_{\mathrm{sb}}(\tilde{\theta}) + \mathcal{L}_{\mathrm{tb}}(\tilde{\theta}) \le C_{\mathrm{app}}\delta^2,
\]
where $C_{\mathrm{app}}>0$ is the \emph{approximation constant}. It may depend on the domain, PDE coefficients, and network classes, but is independent of $\delta$. The tildes distinguish these comparison parameters from the trained parameters $(\theta^*,\eta^*)$; the optimizer is not assumed to produce $(\tilde\theta,\tilde\eta)$.
\end{assumption}

\begin{remark}
The two norm inequalities express Sobolev approximation of the prescribed state and control; classical approximation results for smooth activation functions support this part~\cite{Hornik1991,Pinkus1999}. The residual estimate is an additional compatibility requirement and does not follow from the displayed $H^1(\Omega_T)$ state approximation alone, particularly when the PDE contains second-order derivatives. Assumption~\ref{ass:express} is used only as an existence argument: in the proof, the global PINNs minimizer is compared with this pair to bound the trained loss from above by the continuous optimal value.
\end{remark}

\begin{assumption}[Source condition]\label{ass:source}
The minimum-energy control $f_{\min}$ lies in the range of the operator $\mathcal{S}^*\mathcal{S}$, i.e.\ $f_{\min}\in\mathcal{R}(\mathcal{S}^*\mathcal{S})$, where $\mathcal{S}^*:L^2(\Omega)\to H^1(\Omega_T)$ is the adjoint defined using the $H^1(\Omega_T)$ inner product on the control space.
\end{assumption}

\begin{remark}[Verification of the source condition for the heat equation]\label{rem:source_verify}
For the 1D heat equation on $\Omega=(0,\pi)$ with zero Dirichlet boundary conditions, the spatial part of the solution operator can be analyzed in the sine basis $\{\sin(kx)\}_{k\ge1}$.
However, membership of $f_{\min}$ in $H^1(\Omega_T)$ alone does not imply the range condition $f_{\min}\in\mathcal{R}(\mathcal{S}^*\mathcal{S})$. The latter imposes additional decay relative to the singular values of the full space-time control-to-terminal-state map and is therefore a genuine smoothness assumption.
For a specified target and control space it can be checked through the corresponding singular-system expansion; it is not claimed here to hold automatically for every attainable heat-equation target.
\end{remark}

\begin{theorem}[Control error estimate, linear case]\label{thm:control_error}
Let $\mathcal{F}$ be linear and let $f_{\min}$ be the minimum-energy control~\eqref{eq:minenergy}.
Suppose Assumptions~\ref{ass:source} and~\ref{ass:express} hold.
Let $(u_{\theta^*}, f_{\eta^*})$ be a global minimizer of the PINNs loss~\eqref{eq:loss} with $0<\alpha\le1$ and set $w_J = 1$ for the analysis.
Let $0<\epsilon\le1$. Assume the continuous residuals are small and the quadrature discrepancies in the $L^2(\Omega)$ target term and the $H^1(\Omega_T)$ regularization term are bounded by $C_{\mathrm{quad}}\sqrt{\epsilon}$ for the trained and comparison networks, where $C_{\mathrm{quad}}>0$ is independent of $\epsilon$ and $\alpha$:
\begin{equation}\label{eq:continuous_residual}
\|\mathcal{R}_{\mathrm{int}}\|_{L^2(\Omega_T)}^2 + \|\mathcal{R}_{\mathrm{sb}}\|_{L^2(\Sigma_T)}^2 + \|\mathcal{R}_{\mathrm{tb}}\|_{L^2(\Omega)}^2 \le \epsilon.
\end{equation}
Then there exist constants $C_H,C_L>0$ (independent of $\epsilon$ and $\alpha$) such that
\[
\|f_{\eta^*} - f_{\min}\|_{H^1(\Omega_T)}
\le C_H\Bigl( \frac{\epsilon^{1/4}}{\sqrt{\alpha}} + \alpha \Bigr),
\]
and, consequently,
\[
\|f_{\eta^*} - f_{\min}\|_{L^2(\Omega_T)}
\le C_L\Bigl( \frac{\epsilon^{1/4}}{\sqrt{\alpha}} + \alpha \Bigr).
\]
If $\alpha$ is chosen proportional to $\epsilon^{1/6}$, the bound becomes $\mathcal{O}(\epsilon^{1/6})$.
\end{theorem}

\begin{proof}
To avoid cumbersome notation, in this proof $\mathcal{L}_{\rm target}$ and $\mathcal{L}_{\rm reg}$ also denote their normalized continuous counterparts when they occur inside an analytic estimate. Each passage from a sampled loss to its continuous counterpart contributes at most $C_{\mathrm{quad}}\sqrt{\epsilon}$ by assumption and is absorbed into the constants below.

Decompose $f_{\eta^*} - f_{\min} = (f_{\eta^*} - f_\alpha) + (f_\alpha - f_{\min})$, where $f_\alpha$ minimizes $\mathcal{J}_\alpha$.
Under Assumption~\ref{ass:source}, $\|f_\alpha - f_{\min}\|_{H^1(\Omega_T)} \le C_1\alpha$~\cite{engl1996regularization}, where $C_1>0$ is the source-condition constant and is independent of $\epsilon$ and $\alpha$.

To bound $\|f_{\eta^*} - f_\alpha\|_{H^1(\Omega_T)}$, we use the strong convexity of $\mathcal{J}_\alpha$ on $H^1(\Omega_T)$.
From~\eqref{eq:growth} we have:
\begin{equation}\label{eq:A}
\frac{\alpha}{2}\|f_{\eta^*} - f_\alpha\|_{H^1(\Omega_T)}^2 \le \mathcal{J}_\alpha(f_{\eta^*}) - \mathcal{J}_\alpha(f_\alpha).
\end{equation}

\medskip
\noindent\textit{Step~1: Upper bound for $\mathcal{J}_\alpha(f_{\eta^*})$.}
Let $u(f_{\eta^*})$ be the exact PDE solution for the control $f_{\eta^*}$.
By the PINNs stability estimate~\cite{Mishra2023}, there is a stability constant $C_{\mathrm{stab}}>0$, independent of $\epsilon$ and $\alpha$, such that
\[
\begin{aligned}
\|u_{\theta^*}(\cdot,T) - \mathcal{S}f_{\eta^*}\|_{L^2(\Omega)}
&= \|u_{\theta^*}(\cdot,T) - u(f_{\eta^*})(\cdot,T)\|_{L^2(\Omega)} \le C_{\mathrm{stab}}\sqrt{\epsilon}.
\end{aligned}
\]

Set $a = \mathcal{S}f_{\eta^*} - u_d$ and $b = u_{\theta^*}(\cdot,T) - u_d$.
We claim that $\|a\|_{L^2(\Omega)}, \|b\|_{L^2(\Omega)} \le M$ for a constant $M>0$ independent of $\epsilon$ and $\alpha\in(0,1]$.
Indeed, global minimality bounds the target loss by the loss of a fixed comparison network, and the assumed quadrature estimate therefore bounds $\|b\|_{L^2(\Omega)}$ independently of $\epsilon$ and $\alpha$. The stability estimate gives $\|a-b\|_{L^2(\Omega)}\le C_{\mathrm{stab}}\sqrt{\epsilon}$, so $\|a\|_{L^2(\Omega)}$ is bounded as well.
Using the norm inequality $|\,\|a\|_{L^2(\Omega)}^2-\|b\|_{L^2(\Omega)}^2\,|\le \|a-b\|_{L^2(\Omega)}(\|a\|_{L^2(\Omega)}+\|b\|_{L^2(\Omega)})$, we obtain:
\[
\begin{aligned}
\bigl|
\|\mathcal{S}f_{\eta^*} - u_d\|_{L^2(\Omega)}^2
- \|u_{\theta^*}(\cdot,T) - u_d\|_{L^2(\Omega)}^2
\bigr|
&\le 2 M C_{\mathrm{stab}}\sqrt{\epsilon} \equiv C_2\sqrt{\epsilon}.
\end{aligned}
\]
Therefore:
\begin{align}
\mathcal{J}_\alpha(f_{\eta^*})
&= \frac12\|\mathcal{S}f_{\eta^*} - u_d\|_{L^2(\Omega)}^2
+ \frac{\alpha}{2}\|f_{\eta^*}\|_{H^1(\Omega_T)}^2 \nonumber\\
&\le \frac12\|u_{\theta^*}(\cdot,T)-u_d\|_{L^2(\Omega)}^2 + \frac{\alpha}{2}\|f_{\eta^*}\|_{H^1(\Omega_T)}^2
+ C_2\sqrt{\epsilon} \nonumber\\
&\le \mathcal{L}_{\mathrm{target}}(\theta^*)
+ \alpha\mathcal{L}_{\mathrm{reg}}(\eta^*)
+ C_2\sqrt{\epsilon},
\label{eq:Jbound}
\end{align}
where we have absorbed the quadrature errors of the discrete norms into the $\sqrt{\epsilon}$ term.

\medskip
\noindent\textit{Step~2: Comparison with the continuous optimum.}
Recall that we set $w_J = 1$ in the loss~\eqref{eq:loss}.
Because the residuals $\mathcal{L}_{\mathrm{int}},\mathcal{L}_{\mathrm{sb}},\mathcal{L}_{\mathrm{tb}}$ are non-negative:
\begin{equation}\label{eq:opt}
\mathcal{L}_{\mathrm{target}}(\theta^*) + \alpha\mathcal{L}_{\mathrm{reg}}(\eta^*)
\le \mathcal{L}(\theta^*,\eta^*)
\le \mathcal{L}(\tilde{\theta},\tilde{\eta})
\end{equation}
for any choice of network parameters $(\tilde{\theta},\tilde{\eta})$.

Now define $u_\alpha:=u(f_\alpha)$ and select $(\tilde{\theta},\tilde{\eta})$ to approximate the continuous optimal pair $(u_\alpha,f_\alpha)$.
By parabolic regularity, $u_\alpha \in H^2(\Omega_T)$ and $f_\alpha \in H^1(\Omega_T)$.
Take $\delta = \sqrt{\epsilon}$ in Assumption~\ref{ass:express}; there exist $(\tilde{\theta},\tilde{\eta})$ such that:
\[
\|u_{\tilde{\theta}} - u_\alpha\|_{H^1(\Omega_T)} \le \sqrt{\epsilon},\qquad
\|f_{\tilde{\eta}} - f_\alpha\|_{H^1(\Omega_T)} \le \sqrt{\epsilon},
\]
and the residuals satisfy:
\[
\mathcal{L}_{\mathrm{int}}(\tilde{\theta},\tilde{\eta}) + \mathcal{L}_{\mathrm{sb}}(\tilde{\theta}) + \mathcal{L}_{\mathrm{tb}}(\tilde{\theta}) \le C_{\mathrm{app}}\,\epsilon.
\]
By the trace inequality $\|v(\cdot,T)\|_{L^2(\Omega)} \le C_{\mathrm{tr}}\|v\|_{H^1(\Omega_T)}$ for $v\in H^1(\Omega_T)$, where $C_{\mathrm{tr}}>0$ is the trace constant, we have:
\[
\|u_{\tilde{\theta}}(\cdot,T) - u_\alpha(\cdot,T)\|_{L^2(\Omega)} \le C_{\mathrm{tr}}\sqrt{\epsilon}.
\]
Hence,
\begin{align}
\mathcal{L}_{\mathrm{target}}(\tilde{\theta})
&= \frac12\|u_{\tilde{\theta}}(\cdot,T) - u_d\|_{L^2(\Omega)}^2 \nonumber\\
&\le \frac12\bigl(\|\mathcal{S}f_\alpha - u_d\|_{L^2(\Omega)} + C_{\mathrm{tr}}\sqrt{\epsilon}\bigr)^2 \nonumber\\
&\le \frac12\|\mathcal{S}f_\alpha - u_d\|_{L^2(\Omega)}^2 + C_3\sqrt{\epsilon}, \label{eq:target_tilde}
\end{align}
where $C_3 = C_{\mathrm{tr}}\|\mathcal{S}f_\alpha - u_d\|_{L^2(\Omega)} + \frac12 C_{\mathrm{tr}}^2$ is independent of $\epsilon$ (we used $\epsilon \le \sqrt{\epsilon}$ for $\epsilon\le 1$).

Similarly, using $\|f_{\tilde{\eta}} - f_\alpha\|_{H^1(\Omega_T)} \le \sqrt{\epsilon}$,
\begin{align}
\mathcal{L}_{\mathrm{reg}}(\tilde{\eta})
&= \frac12\|f_{\tilde{\eta}}\|_{H^1(\Omega_T)}^2 \nonumber\\
&\le \frac12\bigl(\|f_\alpha\|_{H^1(\Omega_T)} + \sqrt{\epsilon}\bigr)^2 \nonumber\\
&\le \frac12\|f_\alpha\|_{H^1(\Omega_T)}^2 + C_4\sqrt{\epsilon}, \label{eq:reg_tilde}
\end{align}
with $C_4 = \|f_\alpha\|_{H^1(\Omega_T)} + \frac12$ bounded independently of $\epsilon$ and uniformly in $\alpha\in(0,1]$ because $f_\alpha \to f_{\min}$ in $H^1(\Omega_T)$.

Now assemble the total loss for $(\tilde{\theta},\tilde{\eta})$:
\begin{align*}
\mathcal{L}(\tilde{\theta},\tilde{\eta})
&= \mathcal{L}_{\mathrm{int}}(\tilde{\theta},\tilde{\eta}) + \mathcal{L}_{\mathrm{sb}}(\tilde{\theta}) + \mathcal{L}_{\mathrm{tb}}(\tilde{\theta}) + \mathcal{L}_{\mathrm{target}}(\tilde{\theta}) + \alpha\mathcal{L}_{\mathrm{reg}}(\tilde{\eta}) \\
&\le C_{\mathrm{app}}\,\epsilon + \Bigl(\frac12\|\mathcal{S}f_\alpha - u_d\|_{L^2(\Omega)}^2 + C_3\sqrt{\epsilon}\Bigr) + \alpha\Bigl(\frac12\|f_\alpha\|_{H^1(\Omega_T)}^2 + C_4\sqrt{\epsilon}\Bigr).
\end{align*}
Since $\epsilon \le \sqrt{\epsilon}$ for $\epsilon\le 1$, we can absorb $C_{\mathrm{app}}\epsilon$ into a term of order $\sqrt{\epsilon}$.
Thus:
\[
\mathcal{L}(\tilde{\theta},\tilde{\eta}) \le \mathcal{J}_\alpha(f_\alpha) + C_5\sqrt{\epsilon},
\qquad C_5 = C_3 + C_4 + C_{\mathrm{app}},
\]
where $C_5$ is independent of $\epsilon$ and remains bounded for $\alpha \in (0,1]$
(we used $\alpha C_4 \le C_4$).

Combining this with~\eqref{eq:opt} and~\eqref{eq:Jbound} gives:
\[
\mathcal{J}_\alpha(f_{\eta^*}) - \mathcal{J}_\alpha(f_\alpha) \le C_6\sqrt{\epsilon},
\qquad C_6 = C_2 + C_5.
\]

\medskip
\noindent\textit{Step~3: Closing the estimate.}
From~\eqref{eq:A} we obtain:
\[
\|f_{\eta^*} - f_\alpha\|_{H^1(\Omega_T)} \le \sqrt{\frac{2C_6}{\alpha}}\;\epsilon^{1/4}.
\]
Combining the last estimate with the source-condition bound gives the stronger control-space estimate
\[
\begin{aligned}
\|f_{\eta^*} - f_{\min}\|_{H^1(\Omega_T)}
&\le \|f_{\eta^*} - f_\alpha\|_{H^1(\Omega_T)}
+ \|f_\alpha - f_{\min}\|_{H^1(\Omega_T)}\\
&\le C_7\,\frac{\epsilon^{1/4}}{\sqrt{\alpha}} + C_1\alpha.
\end{aligned}
\]
Here $C_7=\sqrt{2C_6}$; choosing $C_H=\max\{C_7,C_1\}$ gives the first estimate in the theorem.
Finally, the continuous embedding $H^1(\Omega_T)\hookrightarrow L^2(\Omega_T)$ yields
\[
\|f_{\eta^*} - f_{\min}\|_{L^2(\Omega_T)}
\le C_{\mathrm{emb}}\|f_{\eta^*} - f_{\min}\|_{H^1(\Omega_T)},
\]
so the embedding constant can be absorbed into $C_L$. This completes the proof.
\end{proof}

\begin{remark}[On the global minimizer assumption]\label{rem:global_min}
The hypothesis that PINN training reaches a global minimizer is a standard idealization in theoretical analyses of PINNs~\cite{Shin2020convergence,Mishra2023,zhang2023stability}.
Global-convergence results are available for particular architectures and sufficiently wide-network regimes~\cite{jiang2023global,zhao2025convergence,gao2026gradient}, but they do not automatically apply to the finite networks used here.
In our experiments the training loss decays to values of order $10^{-5}$-$10^{-7}$ (Figures~2a and~4a), which indicates a low-loss solution but does not by itself certify global optimality. The theorem should therefore be read as conditional on the stated global-minimizer assumption.
\end{remark}

\begin{remark}[Non-attainable targets]\label{rem:nonattain_thm}
Theorem~\ref{thm:control_error} assumes that the target $u_d$ is attainable, so that the minimum-energy control $f_{\min}$ exists.
In the numerical experiments of Section~5 the targets are discontinuous step functions, which are not attainable by smooth parabolic flows; consequently $f_{\min}$ does not exist in the classical sense, and the source condition Assumption~\ref{ass:source} is not strictly satisfied.
Nevertheless, the regularized PINNs framework still produces smooth, physically plausible controls. Proposition~\ref{prop:heuristic} offers a formal local explanation only for the nonlinear Burgers case. For the linear heat case, standard regularized least-squares theory explains stability for a non-attainable target, but the convergence estimate to an attainable minimum-energy control stated in Theorem~\ref{thm:control_error} does not apply. The experiments therefore demonstrate performance beyond the precise scope of that theorem.
\end{remark}

\begin{corollary}[Convergence of the state]\label{cor:u_convergence}
Under the same assumptions as in Theorem~\ref{thm:control_error},
the PINN state $u_{\theta^*}$ converges to the true optimal state
$u(f_{\min})$ in $L^2(\Omega_T)$.
More precisely, there exists a constant $C_u>0$ such that:
\[
\|u_{\theta^*} - u(f_{\min})\|_{L^2(\Omega_T)} \le
C_u\Bigl( \frac{\epsilon^{1/4}}{\sqrt{\alpha}} + \alpha \Bigr).
\]
\end{corollary}

\begin{proof}
The proof combines three ingredients:
\begin{enumerate}
  \item \textbf{Lipschitz stability of the forward PDE.}  
  For the linear parabolic equation~\eqref{eq:state_pde}, standard energy estimates (see, e.g., \cite{evans2010partial}) imply that the solution map $f \mapsto u_f$ is Lipschitz continuous from $L^2(\Omega_T)$ to $L^2(\Omega_T)$.  
  Hence, there exists a constant $C_{\mathrm{par}}>0$, depending only on $\Omega$, $T$, and the coefficients of the PDE, such that:
  \[
  \|u(f_1) - u(f_2)\|_{L^2(\Omega_T)} \le C_{\mathrm{par}} \|f_1 - f_2\|_{L^2(\Omega_T)}
  \]
  for all admissible controls $f_1,f_2$.  
  In particular,
  \begin{equation}\label{eq:lip}
  \|u(f_{\eta^*}) - u(f_{\min})\|_{L^2(\Omega_T)}
  \le C_{\mathrm{par}} \|f_{\eta^*} - f_{\min}\|_{L^2(\Omega_T)}.
  \end{equation}

  \item \textbf{Stability of the PINN approximation.}  
  By the generalization error estimates for PINNs applied to parabolic PDEs \cite{Mishra2023}, the network state $u_{\theta^*}$ stays close to the true solution of the PDE with control $f_{\eta^*}$.  
  More precisely, there exists a constant $C_{\mathrm{PINN}}>0$, depending only on $\Omega$, $T$, and the network architecture, such that:
  \begin{equation}\label{eq:pinn}
  \|u_{\theta^*} - u(f_{\eta^*})\|_{L^2(\Omega_T)} \le C_{\mathrm{PINN}} \sqrt{\epsilon}.
  \end{equation}
  (This estimate is the standard ``stability bound'' used throughout the PINN literature; the constant $C_{\mathrm{PINN}}$ absorbs the quadrature errors and the approximation error of the network class.)

  \item \textbf{Error bound for the control.}  
  Theorem~\ref{thm:control_error} provides a quantitative estimate for the deviation of the learned control from the minimum-energy control:
  \begin{equation}\label{eq:control_bound}
  \|f_{\eta^*} - f_{\min}\|_{L^2(\Omega_T)} \le C_L\Bigl(\frac{\epsilon^{1/4}}{\sqrt{\alpha}} + \alpha\Bigr),
  \end{equation}
  where $C_L$ is the constant from Theorem~\ref{thm:control_error}.
\end{enumerate}

Now combine these bounds using the triangle inequality:
\[
\begin{aligned}
\|u_{\theta^*} - u(f_{\min})\|_{L^2(\Omega_T)}
&\le \|u_{\theta^*} - u(f_{\eta^*})\|_{L^2(\Omega_T)} + \|u(f_{\eta^*}) - u(f_{\min})\|_{L^2(\Omega_T)} \\
&\le C_{\mathrm{PINN}} \sqrt{\epsilon} + C_{\mathrm{par}}C_L \Bigl(\frac{\epsilon^{1/4}}{\sqrt{\alpha}} + \alpha \Bigr) \\
&\le C_u \Bigl( \frac{\epsilon^{1/4}}{\sqrt{\alpha}} + \alpha \Bigr),
\end{aligned}
\]
where we used $\sqrt{\epsilon} \le \epsilon^{1/4}$ for $\epsilon\le 1$ and set:
\[
C_u = C_{\mathrm{PINN}} + C_{\mathrm{par}}C_L.
\]
This completes the proof.
\end{proof}

\subsubsection{Heuristic estimate for nonlinear problems}

For nonlinear PDEs such as the Burgers' equation, the control-to-terminal-state map $\mathcal{S}$ is not linear and the Tikhonov functional~\eqref{eq:tikhonov} is non-convex.
Consequently, the strong convexity argument does not apply globally.
Nevertheless, if the learned control $f_{\eta^*}$ is sufficiently close to the continuous minimizer $f_\alpha$, one can linearize the problem around $f_\alpha$ and apply the previous analysis formally.
This leads to the following heuristic estimate, which serves as a guideline rather than a rigorous theorem.

\begin{proposition}[Nonlinear case, formal estimate]\label{prop:heuristic}
Let $\mathcal{F}$ be a nonlinear operator that is Fr\'{e}chet differentiable with respect to $u$ (e.g.\ Burgers' equation, see~\cite{troltzsch2001sqp}).
Assume that the target $u_d$ is attainable. Let $\mathcal{S}'(f_\alpha):H^1(\Omega_T)\to L^2(\Omega)$ denote the Fr\'{e}chet derivative of the control-to-terminal-state map at $f_\alpha$, and let $\mathcal{S}'(f_\alpha)^*$ denote its adjoint with respect to the $H^1(\Omega_T)$ control-space inner product. Assume the analogous source condition
\[
f_{\min}\in\mathcal{R}\!\left(\mathcal{S}'(f_\alpha)^*\mathcal{S}'(f_\alpha)\right).
\]
Suppose in addition that the PINNs minimizer $f_{\eta^*}$ lies in a neighborhood of $f_\alpha$ where the linear approximation is valid, and that the continuous residual level $\epsilon$ in~\eqref{eq:continuous_residual} is small enough.
Then, formally, there is a local constant $C_{\mathrm{nl}}>0$, independent of $\epsilon$ and $\alpha$, such that the bound
\[
\|f_{\eta^*} - f_{\min}\|_{L^2(\Omega_T)} \le C_{\mathrm{nl}}\left(\frac{\epsilon^{1/4}}{\sqrt{\alpha}} + \alpha\right)
\]
is expected; $C_{\mathrm{nl}}$ may depend on the linearization error and the chosen neighborhood of $f_\alpha$.
\end{proposition}

\begin{proof}[Sketch of justification]
For $\|f-f_\alpha\|_{H^1(\Omega_T)}$ sufficiently small:
\[
\mathcal{S}(f) \approx \mathcal{S}(f_\alpha) + \mathcal{S}'(f_\alpha)(f-f_\alpha).
\]
Substituting this into $\mathcal{J}_\alpha$ yields a locally strongly convex quadratic functional in $f-f_\alpha$.
Repeating the steps of Theorem~\ref{thm:control_error} within this local region gives the claimed dependence.
A rigorous justification would require a quantitative bound on the nonlinear remainder and a proof that $f_{\eta^*}$ indeed stays inside the region of convexity; this is beyond the scope of the present work.
\end{proof}

\begin{remark}
Proposition~\ref{prop:heuristic} indicates that the error structure $\mathcal{O}(\epsilon^{1/4}/\sqrt{\alpha} + \alpha)$ is robust, at least in a neighborhood of the optimum.
This is consistent with the numerical results for the Burgers' equation (Section~\ref{sec:numerics}), where the regularized PINNs converges to a smooth, physically plausible control even though a rigorous theory is not yet available.
\end{remark}

    \section{Numerical Experiments}
\label{sec:numerics}

We validate the proposed framework on two benchmark problems: the linear heat equation and the nonlinear Burgers' equation. For each, we first train a standard PINN on the forward problem to identify a suitable architecture, then solve the regularized optimal control problem with both $u$ and $f$ parameterized as independent networks. The regularization parameter $\alpha$ is selected via the L-curve criterion, and the selected value is denoted by $\alpha^*$; the weight $w_J$ is determined by the one-dimensional line search described in Section~3. For the forward tests, the reported relative error is the space-time quantity $E_{\mathrm{rel}}:=\|u_\theta-u_{\mathrm{exact}}\|_{L^2(D)}/\|u_{\mathrm{exact}}\|_{L^2(D)}$, where $D$ is the space-time domain of the corresponding experiment and $u_{\mathrm{exact}}$ is its analytical solution; the associated pointwise absolute error is $|u_\theta-u_{\mathrm{exact}}|$. In the control experiments, the \emph{rollout state} is the numerical PDE solution obtained using the learned control $f_{\eta^*}$, as distinct from the independently parameterized state network $u_{\theta^*}$; the pointwise surrogate error is $|u_{\theta^*}-u(f_{\eta^*})|$.

\subsection{Heat Equation}

In the one-dimensional examples below, $\partial_x$ and $\partial_x^2$ denote the first and second derivatives with respect to the spatial variable $x$, respectively. Consider the heat equation:
\begin{equation}\label{eq:heat_pde}
    \begin{cases}
        \partial_t u - \partial_x^2 u = f(x,t), & (x,t) \in (0,\pi) \times (0,1), \\
        u(0,t) = u(\pi,t) = 0, & t \in (0,1), \\
        u(x,0) = \sin(x), & x \in (0,\pi).
    \end{cases}
\end{equation}

\subsubsection{Forward Problem}
Setting $f \equiv 0$, the exact solution is $u_{\mathrm{exact}}(x,t) = \sin(x)e^{-t}$. The state network uses 4 hidden layers $\times$ 50 neurons with $\tanh$ activation. Training uses $N_{\text{int}} = 10{,}000$, $N_{\text{sb}} = 80$, $N_{\text{tb}} = 40$ collocation points, Adam optimizer with initial learning rate $10^{-3}$ decayed by $10\times$ every 3000 epochs over 6000 total epochs. The relative error $E_{\mathrm{rel}}$ on $D=(0,\pi)\times(0,1)$ decreases steadily and the predicted solution closely matches the exact solution (Figure~\ref{fig:forward_heat}), confirming the architecture is well-suited for this problem.

\begin{figure*}[ht!]
    \centering
    \includegraphics[width=\textwidth]{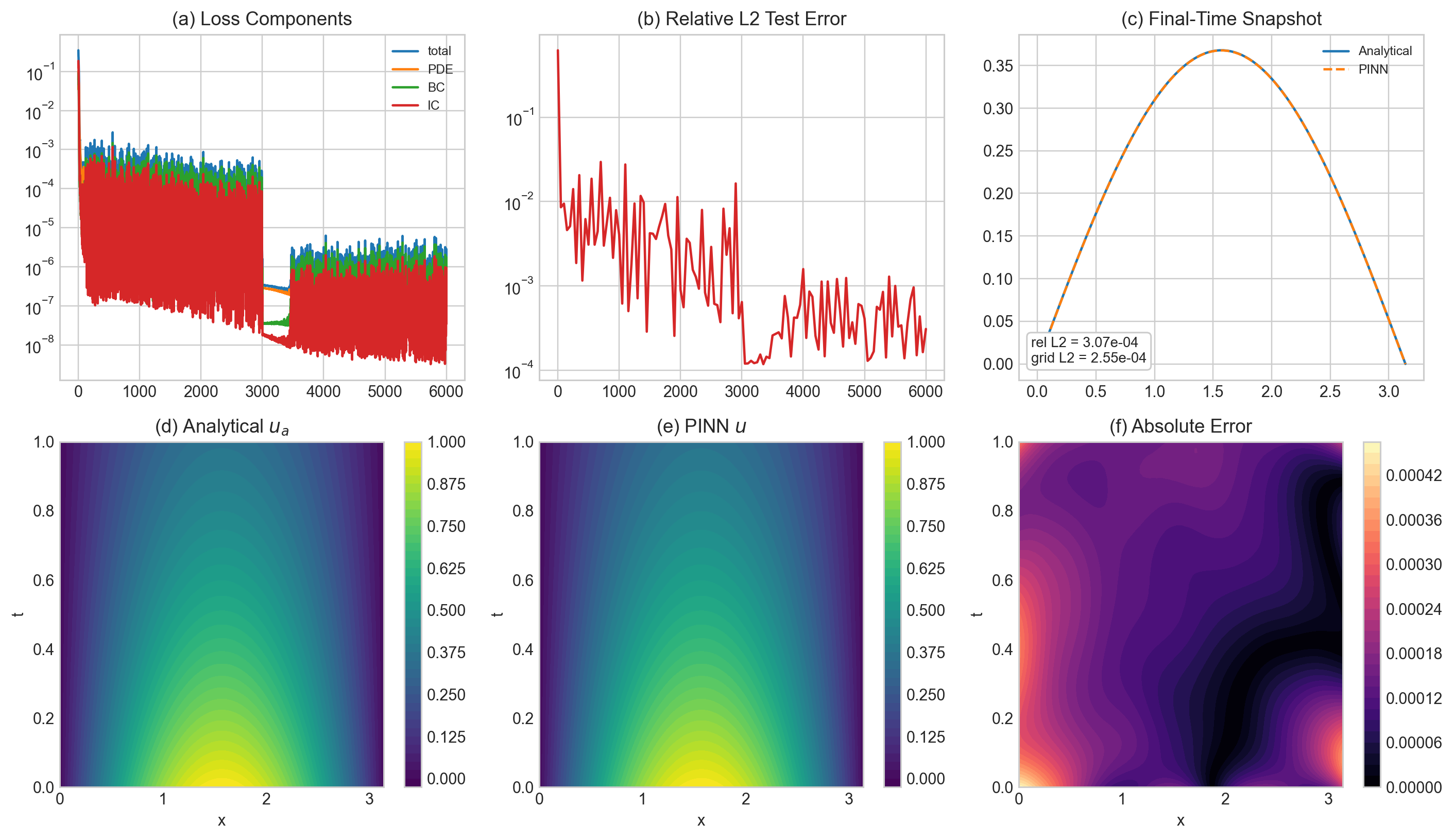}
    \caption{Forward problem results for the heat equation: (a) Loss components; (b) Relative error $E_{\mathrm{rel}}$ on $D=(0,\pi)\times(0,1)$; (c) Predicted vs.\ exact solution at $t=1$; (d) Exact solution; (e) Predicted solution; (f) Pointwise absolute error.}
    \label{fig:forward_heat}
\end{figure*}

\subsubsection{Optimal Control Problem}

The desired terminal state is the discontinuous step function:
\[
u_d(x) = \begin{cases} 1, & x \in [\tfrac{\pi}{3},\tfrac{2\pi}{3}], \\ 0, & \text{otherwise.} \end{cases}
\]
The same architecture and collocation counts are used for both $u_\theta$ and $f_\eta$, with $N_d = N_{\text{reg}} = 100$. The Adam optimizer runs for 10,000 epochs with learning rate $10^{-3}$ decayed $10\times$ at epoch 5000.

Figure~\ref{fig:optimal_heat_overview} summarizes the results. All loss components decay stably (panel~a). Terminal tracking is accurate in both the regularized and unregularized cases (panels~b-c). The key difference appears in panel~(d): the unregularized control ($\alpha=0$) exhibits high-frequency oscillations exceeding $\pm 60$, whereas the regularized control is smooth and physically plausible, confirming that the $H^1(\Omega_T)$ penalty favors a lower-energy control. The L-curve (panel~e) identifies $\alpha^* = 3.16\times10^{-3}$ at the point of maximum curvature; the line search (panel~f) selects $w_J$ by minimizing the rollout terminal objective. Panels~(g)-(i) show that the pointwise error between the rollout state and the PINNs prediction stays below $2.2\times10^{-2}$, with the largest errors near the final time where the discontinuous target induces sharper gradients.

\begin{figure*}[ht!]
    \centering
    \includegraphics[width=\textwidth]{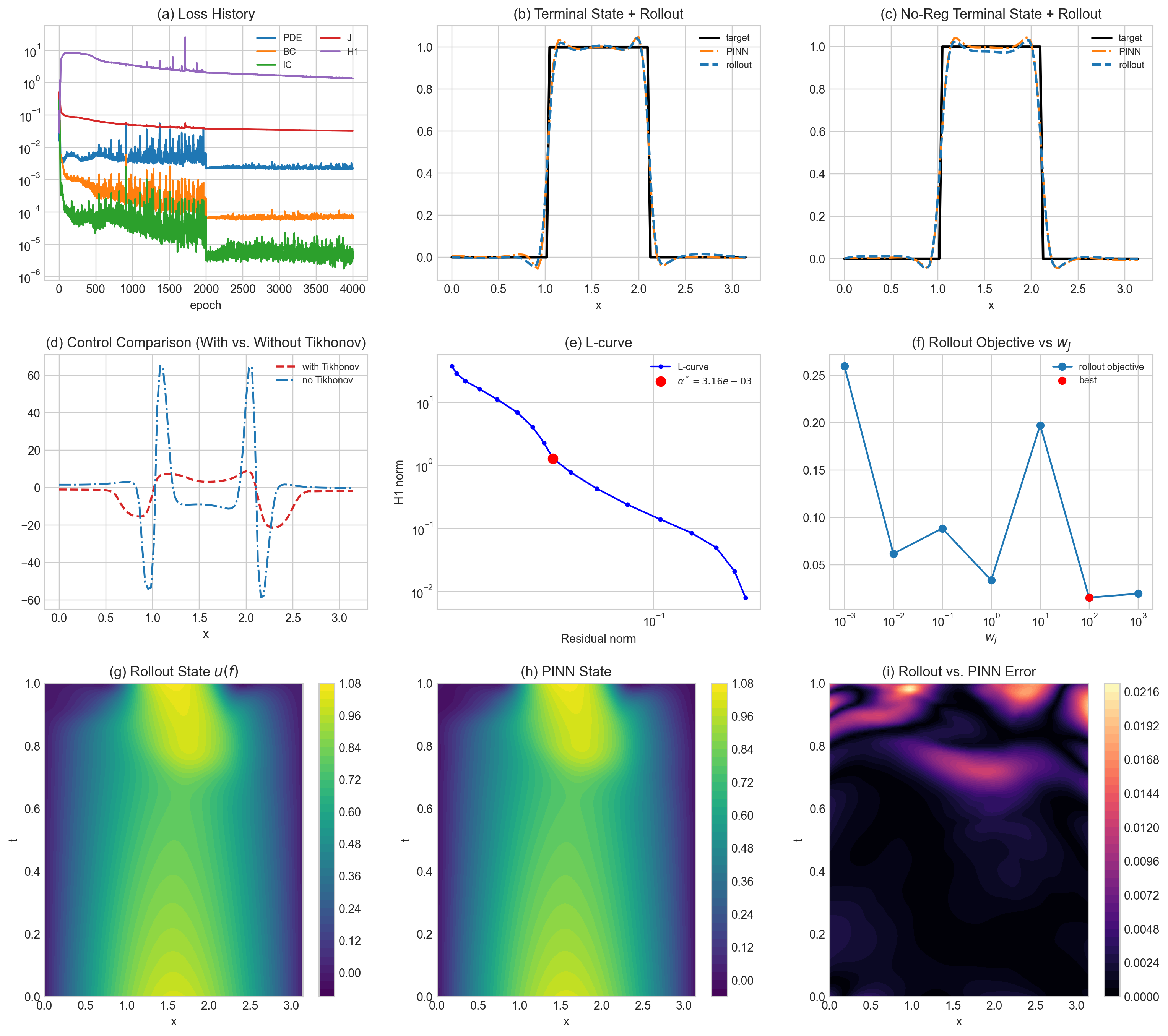}
    \caption{Optimal control results for the heat equation with discontinuous target $u_d$.
    (a) Loss convergence history.
    (b) Terminal state comparison (regularized): PINNs prediction $u_{\theta^*}(\cdot,T)$ and rollout vs.\ $u_d$.
    (c) Same for unregularized ($\alpha=0$).
    (d) Control profiles $f_{\eta^*}(x,1)$ with and without regularization.
    (e) L-curve; optimal $\alpha^* = 3.16\times10^{-3}$.
    (f) One-dimensional line search over $w_J$ using the rollout terminal objective.
    (g) Rollout state $u(f_{\eta^*})$.
    (h) PINNs state $u_{\theta^*}$.
    (i) Pointwise error; maximum $2.2\times10^{-2}$.}
    \label{fig:optimal_heat_overview}
\end{figure*}

Figure~\ref{fig:heat_control_surfaces} displays the control surfaces. The regularized control is smooth and of moderate amplitude, concentrated near $t\approx1$ and the discontinuities of $u_d$ at $x=\pi/3$ and $x=2\pi/3$. The unregularized control shows spikes exceeding $\pm60$ near final time, confirming the stabilizing role of the $H^1(\Omega_T)$ penalty.

\begin{figure*}[ht!]
    \centering
    \includegraphics[width=0.8\textwidth]{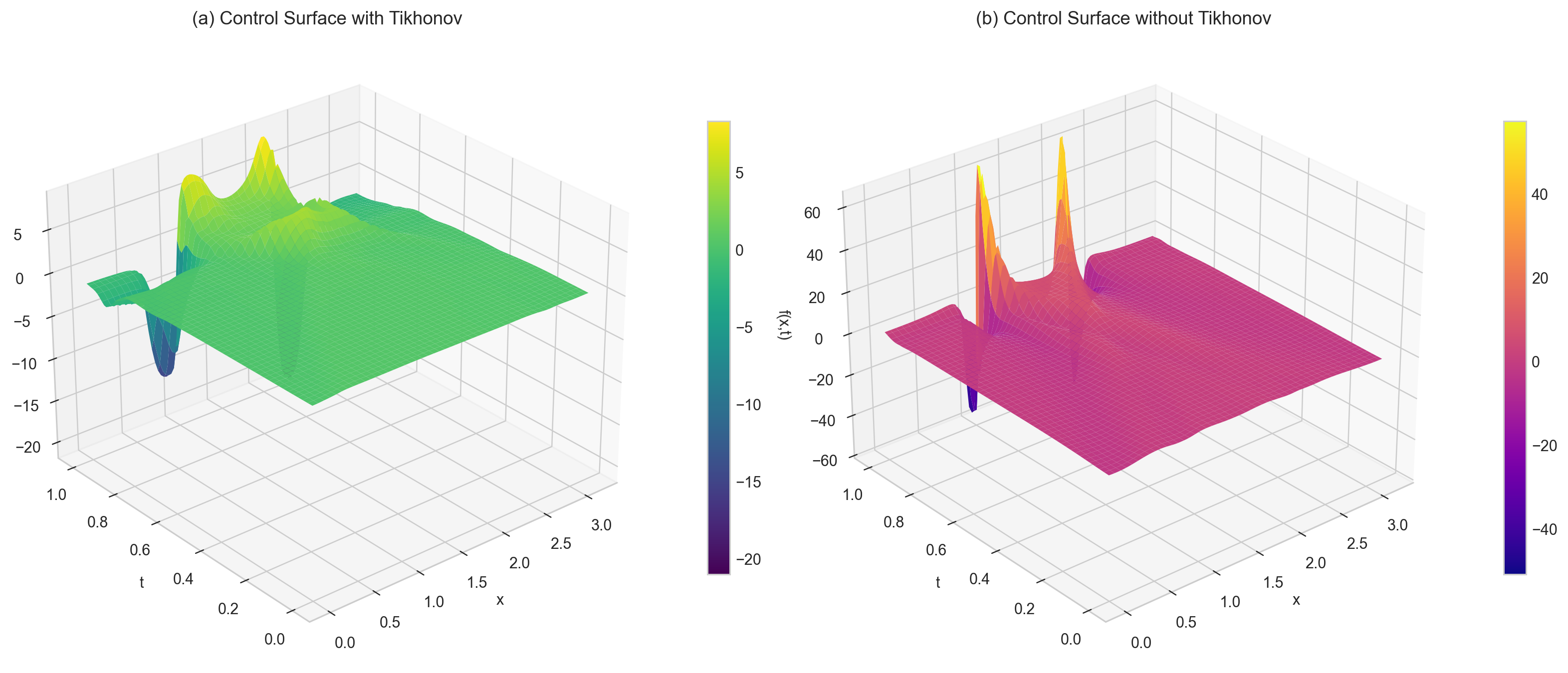}
    \caption{Control surfaces $f_{\eta^*}(x,t)$ for the heat equation:
    (a) regularized ($\alpha^* = 3.16\times10^{-3}$);
    (b) unregularized ($\alpha=0$).}
    \label{fig:heat_control_surfaces}
\end{figure*}

\subsection{Burgers' Equation}

We further evaluate the method on the nonlinear Burgers' equation with viscosity $\nu=0.1$:
\begin{equation}\label{eq:burgers_pde}
    \begin{cases}
        \partial_t u + u\,\partial_x u - \nu\partial_x^2 u = f(x,t), & (x,t) \in (0,1)^2, \\
        u(0,t) = u(1,t) = 0, & t \in (0,1), \\
        u(x,0) = \dfrac{2\nu\pi \sin(\pi x)}{2 + \cos(\pi x)}, & x \in (0,1),
    \end{cases}
\end{equation}

\subsubsection{Forward Problem}
The exact solution for $f\equiv0$ is:
\begin{equation}\label{eq:burgers_analytical}
    u_{\mathrm{exact}}(x,t) = \dfrac{2\nu\pi \sin(\pi x)\,e^{-\nu\pi^2 t}}{2 + \cos(\pi x)\,e^{-\nu\pi^2 t}}.
\end{equation}
The same state-network architecture ($4\times50$, $\tanh$) and collocation counts are used, with 20,000 Adam epochs (learning rate decayed $10\times$ at epoch 10,000). Figure~\ref{fig:forward_burgers} shows reliable convergence: the relative error $E_{\mathrm{rel}}$ on $D=(0,1)^2$ decreases steadily and the predicted solution closely tracks the exact solution across the full space-time domain.

\begin{figure*}[ht!]
    \centering
    \includegraphics[width=\textwidth]{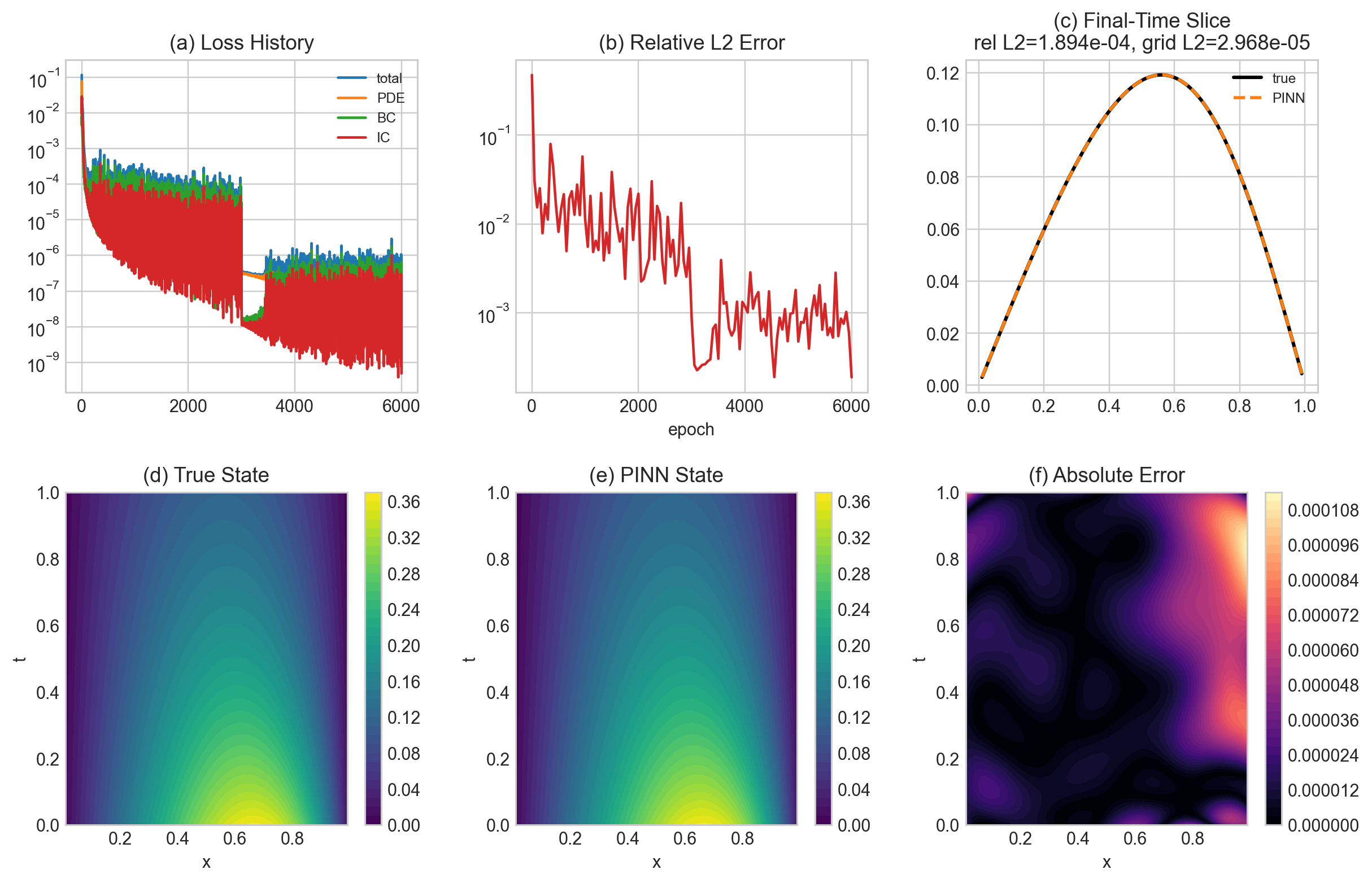}
    \caption{Forward problem results for the Burgers' equation: (a) Loss components; (b) Relative error $E_{\mathrm{rel}}$ on $D=(0,1)^2$; (c) Comparison at $t=1$; (d) Exact solution; (e) Predicted solution; (f) Pointwise absolute error.}
    \label{fig:forward_burgers}
\end{figure*}

\subsubsection{Optimal Control Problem}

The desired terminal state is:
\[
u_d(x) = \begin{cases} 1, & x \in [0.3,0.7], \\ 0, & \text{otherwise.} \end{cases}
\]
The same architecture, collocation counts ($N_{\text{int}}=10{,}000$, $N_{\text{sb}}=80$, $N_{\text{tb}}=40$, $N_d=N_{\text{reg}}=100$), and hyperparameter selection procedure are used, with 20,000 Adam epochs.

Figure~\ref{fig:optimal_burgers_overview} presents the results. All loss components converge stably, with a pronounced drop after epoch 10,000 (learning rate decay); the interior PDE-residual loss $\mathcal{L}_{\mathrm{int}}$ reaches $10^{-4}$-$10^{-5}$ at convergence. Terminal tracking is accurate in both cases (panels~b-c), with natural smoothing near $x\approx0.3$ and $x\approx0.7$ due to the parabolic operator. Panel~(d) shows the key effect: at $t=1$, the unregularized control has sharp spikes ranging approximately from $-50$ to $60$ near the discontinuities, while the regularized control is significantly smoother. The L-curve (panel~e) yields $\alpha^*=5.62\times10^{-5}$, smaller than the heat equation value, reflecting that the nonlinear convection term concentrates control energy near the jump locations so the $H^1(\Omega_T)$ penalty must be weaker to allow necessary localized actuation. The pointwise surrogate error (panels~g-i) stays below $5.4\times10^{-3}$.

\begin{figure*}[ht!]
    \centering
    \includegraphics[width=\textwidth]{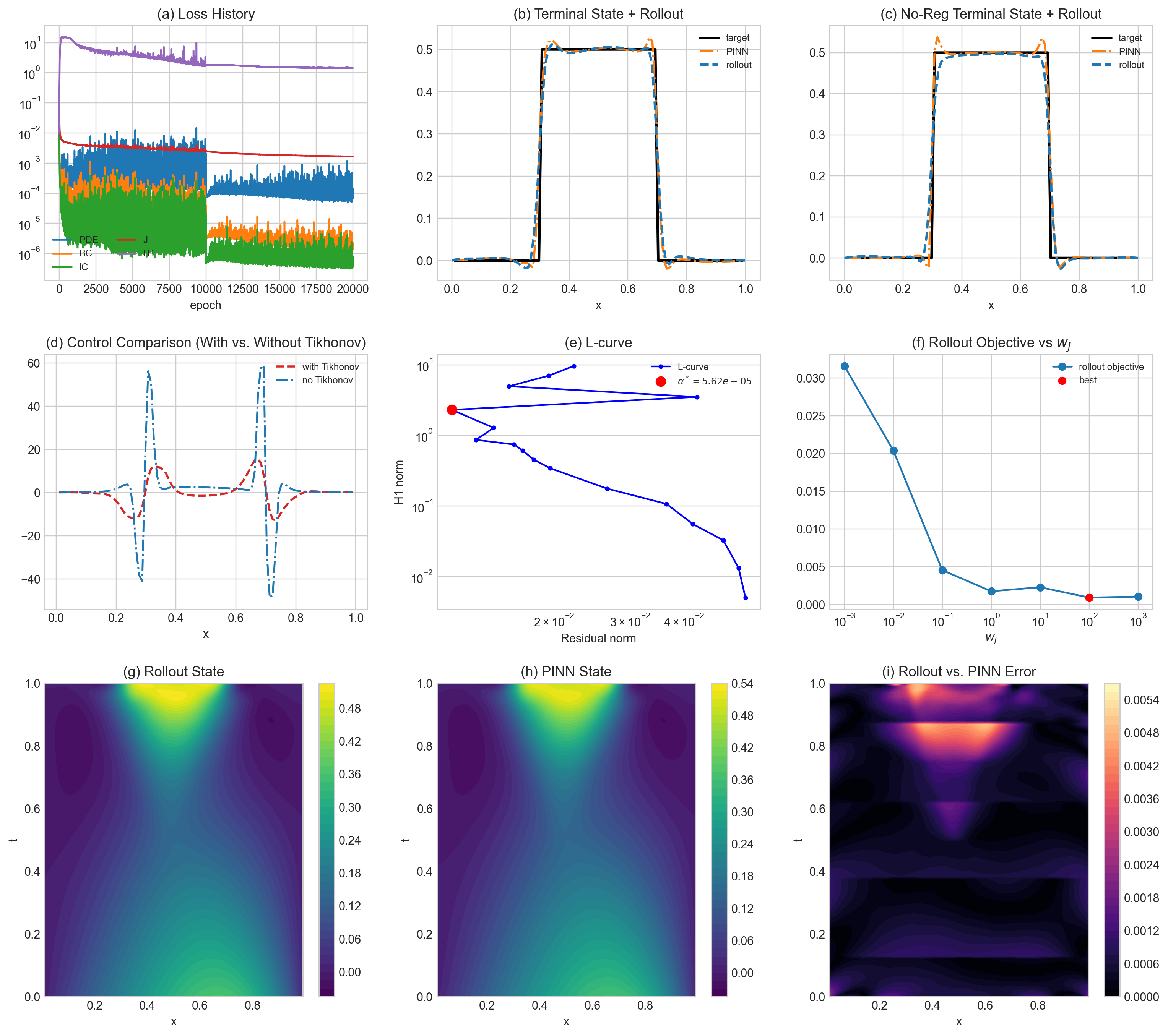}
    \caption{Optimal control results for the Burgers' equation with step-function target $u_d$.
    (a) Loss convergence over 20,000 epochs.
    (b) Terminal state comparison (regularized).
    (c) Same for unregularized ($\alpha=0$).
    (d) Control profiles with and without regularization.
    (e) L-curve; optimal $\alpha^* = 5.62\times10^{-5}$.
    (f) One-dimensional line search over $w_J$ using the rollout terminal objective.
    (g) Rollout state $u(f_{\eta^*})$.
    (h) PINNs state $u_{\theta^*}$.
    (i) Pointwise error; maximum $5.4\times10^{-3}$.}
    \label{fig:optimal_burgers_overview}
\end{figure*}

Figure~\ref{fig:burgers_control_surfaces} displays the control surfaces over the full time interval. The regularized control concentrates energy near $x=0.3$ and $x=0.7$ with moderate amplitude (approximately $-13$ to $15$), while the unregularized control reaches roughly $-50$ to $60$ near the same spatial locations. Compared with Figure~\ref{fig:heat_control_surfaces}, the Burgers control is more spatially localized, consistent with the smaller $\alpha^*$.

\begin{figure*}[ht!]
    \centering
    \includegraphics[width=0.8\textwidth]{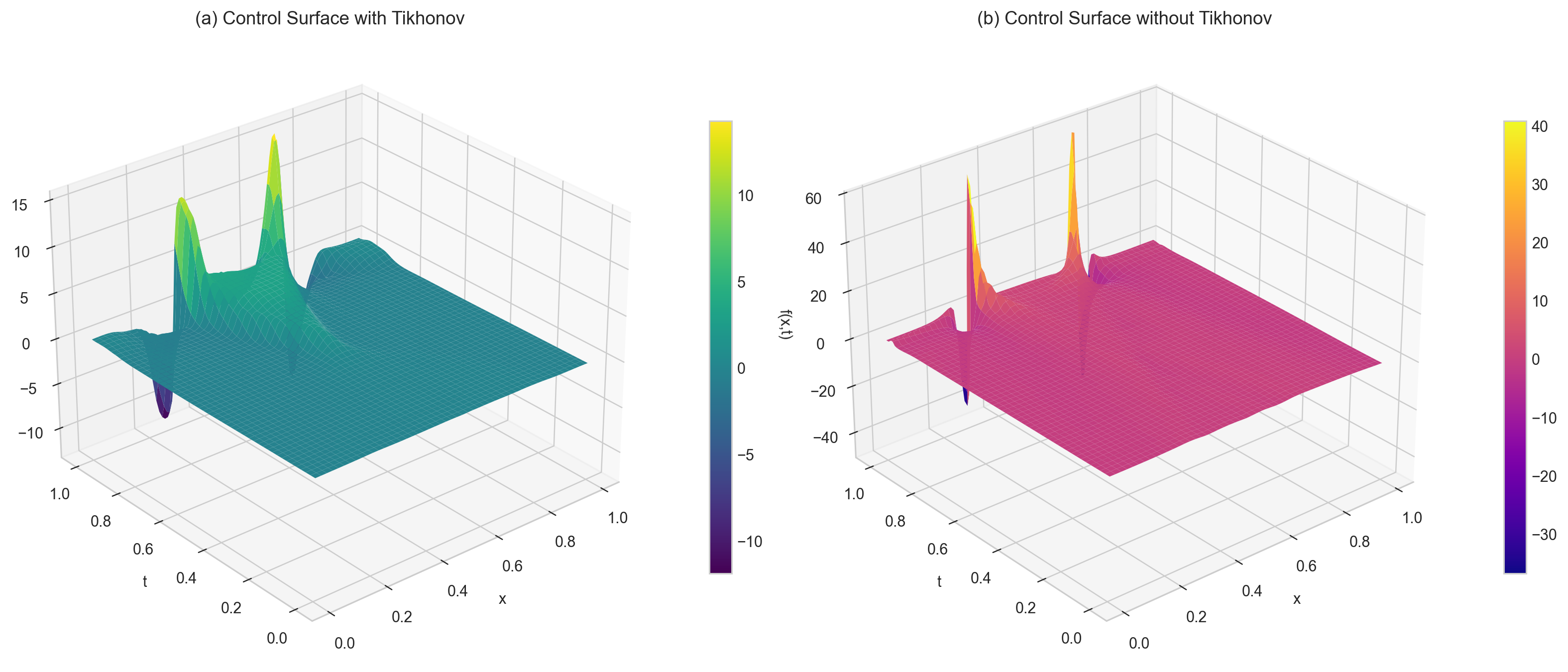}
    \caption{Control surfaces $f_{\eta^*}(x,t)$ for the Burgers' equation:
    (a) regularized ($\alpha^* = 5.62\times10^{-5}$);
    (b) unregularized ($\alpha=0$).}
    \label{fig:burgers_control_surfaces}
\end{figure*}

    \section{Conclusion}
\label{sec:conclusion}

We have proposed a Tikhonov-regularized PINNs framework for terminal-state tracking optimal control of parabolic PDEs, addressing the inherent ill-posedness from non-uniqueness of optimal controls. On the theoretical side, we established a consistency result (Proposition~\ref{prop:consistency}) and a novel error estimate (Theorem~\ref{thm:control_error}); a formal extension to the nonlinear case is given in Proposition~\ref{prop:heuristic}. Numerical experiments on the heat and Burgers' equations confirm that the regularized framework yields smooth, physically meaningful controls without sacrificing terminal tracking accuracy.

Future work includes extending the framework to boundary control and more complex systems~\cite{thi2024computational}, incorporating adaptive loss-weighting to reduce manual tuning, and developing a rigorous error analysis for the nonlinear case along the lines of~\cite{zhang2023stability}.
    \bibliographystyle{ieeetr} 
    \bibliography{ref}
\end{document}